\documentclass[english]{article}
\usepackage{fourier}
\usepackage[T1]{fontenc}
\usepackage[latin9]{inputenc}
\usepackage[a4paper]{geometry}
\usepackage{babel}
\usepackage{amsthm}
\usepackage{amsmath}
\usepackage{amssymb}

\makeatletter
\usepackage{enumitem}		
\theoremstyle{plain}
\newtheorem{thm}{\protect\theoremname}
  \theoremstyle{plain}
  \newtheorem{prop}[thm]{\protect\propositionname}
  \theoremstyle{remark}
  \newtheorem*{rem*}{\protect\remarkname}
  \theoremstyle{plain}
  \newtheorem{lem}[thm]{\protect\lemmaname}
 \theoremstyle{remark}
 \newtheorem*{defn*}{\protect\definitionname}
  \theoremstyle{remark}
  \newtheorem{rem}[thm]{\protect\remarkname}
  \theoremstyle{plain}
  \newtheorem{claim}[thm]{\protect\claimname}
  \theoremstyle{plain}
  \newtheorem{conjecture}[thm]{\protect\conjecturename}

\@ifundefined{date}{}{\date{}}

\@ifundefined{showcaptionsetup}{}{%
 \PassOptionsToPackage{caption=false}{subfig}}
\usepackage{subfig}
\makeatother

  \providecommand{\claimname}{Claim}
  \providecommand{\conjecturename}{Conjecture}
  \providecommand{\definitionname}{Definition}
  \providecommand{\lemmaname}{Lemma}
  \providecommand{\propositionname}{Proposition}
  \providecommand{\remarkname}{Remark}
\providecommand{\theoremname}{Theorem}

\makeatletter \renewenvironment{proof}[1][\proofname] {\par\pushQED{\qed}\normalfont\topsep6\p@\@plus6\p@\relax\trivlist\item[\hskip\labelsep\bfseries#1\@addpunct{.}]\ignorespaces}{\popQED\endtrivlist\@endpefalse} \makeatother

\begin{document}

\title{Component Games on Regular Graphs}

\author{Rani Hod%
\thanks{School of Computer Science, Raymond and Beverly Sackler Faculty of
Exact Sciences, Tel Aviv University, Tel Aviv, Israel. E-mail: rani.hod@cs.tau.ac.il.
Research supported by an ERC advanced grant.%
} \and Alon Naor%
\thanks{School of Mathematical Sciences, Raymond and Beverly Sackler Faculty
of Exact Sciences, Tel Aviv University, Tel Aviv, Israel. E-mail:
alonnaor@tau.ac.il.%
}}
\maketitle
\begin{abstract}
We study the $\left(1:b\right)$ Maker--Breaker component game, played
on the edge set of a $d$-regular graph. \textsc{Maker}'s aim in this
game is to build a large connected component, while \textsc{Breaker}'s
aim is to not let him do so. For all values of \textsc{Breaker}'s
bias $b$, we determine whether \textsc{Breaker} wins (on any $d$-regular
graph) or \textsc{Maker} wins (on almost every $d$-regular graph)
and provide explicit winning strategies for both players.

To this end, we prove an extension of a theorem by Gallai--Hasse--Roy--Vitaver
about graph orientations without long directed simple paths.
\end{abstract}

\section{Introduction}

Let $X$ be a finite set, let $\mathcal{F}\subseteq2^{X}$ be a family
of subsets of $X$, and let $m,b$ be two positive integers. In the
$\left(m:b\right)$ Maker--Breaker game $\left(X,\mathcal{F}\right)$,
two players, called \textsc{Maker} and \textsc{Breaker}, take turns
in claiming previously unclaimed elements of $X$. On \textsc{Maker}'s
move, he claims $m$ elements of $X$, and on \textsc{Breaker}'s move,
he claims $b$ elements.%
\footnote{The player who makes the very last move may not be able to complete
$m$ (or $b$) steps, so he stops after claiming all remaining elements.%
} The game ends when all of the elements have been claimed by either
of the players. The description of the game is complete by stating
which of the players is the first to move, though usually it makes
little difference.\textsc{ Maker} wins the game $\left(X,\mathcal{F}\right)$
if by the end of the game he has claimed all the elements of some
$F\in\mathcal{F}$; otherwise \textsc{Breaker} wins.%
\footnote{For convenience, we typically assume that $\mathcal{F}$ is closed
upwards, and specify only the inclusion-minimal elements of $\mathcal{F}$. %
} Since these are finite, perfect information games with no possibility
of draw, for each setup of $\mathcal{F},m,b$ and the identity of
the first player, one of the players has a strategy to win regardless
of the other player's strategy. Therefore, for a given game we may
say that the game is \textsc{Maker}\textquoteright{}s win, or alternatively
that it is \textsc{Breaker}\textquoteright{}s win. The set $X$ is
referred to as the \emph{board} of the game, and the elements of $\mathcal{F}$
are referred to as the \emph{winning sets}. 

When $m=b=1$, we say that the game is \emph{unbiased}; otherwise
it is \emph{biased}, and the positive integers $m$ and $b$ are called
the \emph{bias} of \textsc{Maker} and \textsc{Breaker}, respectively.
Maker--Breaker games are \emph{bias monotone}. It means that if \textsc{Maker}
wins some game with bias $(m:b)$, he also wins this game with bias
$(m':b')$, for every $m'\geq m$, $b'\leq b$. Similarly, if \textsc{Breaker}
wins a game with bias $(m:b)$, he also wins this game with bias $(m':b')$,
for every $m'\leq m$, $b'\geq b$. Indeed, suppose that some player
has a winning strategy with bias $c$, and now he plays with bias
$c'>c$. He can use his old strategy and in addition claim arbitrarily
$c'-c$ extra elements per move and pretend he did not claim them;
whenever his strategy tells him to claim some element he has previously
claimed he just claims arbitrarily some unclaimed element. Similarly,
if his opponent claims less elements, he can assign (in his mind)
some extra elements to his opponent in each move, and continue with
his strategy. The same reasoning shows that it is never a disadvantage
in a Maker--Breaker game to be the first player, and a winning strategy
as a second player can be used as a winning strategy as a first player.
This bias monotonicity allows us to define the \emph{threshold bias}:
for a given game $\mathcal{F}$, the threshold bias $b^{*}$ is the
value for which \textsc{Maker} wins the game $\mathcal{F}$ with bias
$(1:b)$ for every $b\leq b^{*}$, and \textsc{Breaker} wins the game
$\mathcal{F}$ with bias $(1:b)$ for every $b>b^{*}$. 

In this paper, our attention is dedicated to the $\left(1:b\right)$
Maker--Breaker $s$-component game on regular graphs; that is, the
board is the edge set of some $d$-regular graph $G$ on $n$ vertices
and the winning sets are connected components of $G$ with $s$ vertices.

\subsection{Previous results}

A natural case to consider is $s=n$; that is, the winning sets are
the spanning trees of $G$. This $\left(1:b\right)$ $n$-component
game is also known as the \emph{connectivity} game.

The unbiased game was completely solved by Lehman~\cite{Lehman-64},
who showed that \textsc{Maker} wins the $\left(1:1\right)$ connectivity
game on a graph $G$ if and only if $G$ contains two edge disjoint
spanning trees. It follows easily from~\cite{Nash-61,Tutte-61} that
if $G$ is $2k$-edge-connected then it contains $k$ pairwise independent
spanning trees; thus, \textsc{Maker} wins the $\left(1:1\right)$
connectivity game on $4$-regular $4$-edge-connected graphs, whereas
\textsc{Breaker} trivially wins the $\left(1:1\right)$ connectivity
game on graphs with less than $2n-2$ edges, i.e., average degree
under $4-O\left(1/n\right)$.

For denser graphs, since \textsc{Maker} wins the unbiased game by
such a large margin, it only seems fair to even out the odds by strengthening
\textsc{Breaker}, giving him a bias $b\ge2$. First and most natural
board to consider is the edge set of the complete graph $K_{n}$ (i.e.,
$d=n-1$). Chv\'atal and Erd\H{o}s~\cite{CH-78} showed that $\left(\frac{1}{4}-o\left(1\right)\right)n/\log n\le b^{*}\left(K_{n}\right)\le\left(1+o\left(1\right)\right)n/\log n$;
the upper bound was proved to be tight by Gebauer and Szab\'o~\cite{GS-09};
that is, \textsc{$b^{*}\left(K_{n}\right)=\left(1+o\left(1\right)\right)n/\log n$.}
The doubly-biased connectivity game $\left(m:b\right)$ on $K_{n}$
was considered by Hefetz et al.~\cite{HMS-12}, where the winner
was determined for almost all values of $m$ and $b$.

Another natural board to consider is the edge set of a random graph.
Stojakovi\'{c} and Szab\'o~\cite{SS-05} considered the well known
Erd\H{o}s--R\'enyi random graph $\mathcal{G}_{n,p}$, in which each
of the ${n \choose 2}$ possible edges appears independently with
probability $p$. They showed that almost surely $b^{*}\left(\mathcal{G}_{n,p}\right)=\Theta\left(np/\log n\right)$,
where \textsc{Breaker}'s win holds for any $0\le p\le1$ while \textsc{Maker}'s
win requires $p\ge\left(1+o\left(1\right)\right)\log n/n$ for $\mathcal{G}_{n,p}$
to be foremost connected. A different random graph model, the random
$d$-regular graph $\mathcal{G}_{n,d}$ on $n$ vertices, was considered
by Hefetz et al.~\cite{HKSS-11}. They showed that almost surely
$b^{*}\left(\mathcal{G}_{n,d}\right)\ge\left(1-\epsilon\right)d/\log_{2}n$
for $d=o\left(\sqrt{n}\right)$.%
\footnote{By concentration of the binomial distribution, when $d=\Omega\left(\sqrt{n}\right)$,
$\mathcal{G}_{n,d}$ is quite close to $\mathcal{G}_{n,p}$ for $p=d/n$.%
} Moreover, they showed that $b^{*}\left(G\right)\le\max\left\{ 2,\bar{d}/\log n\right\} $
for a graph $G$ of average degree $\bar{d}$, so the result is asymptotically
tight.

\bigskip{}
\textsc{Breaker}'s strategy in practically all results mentioned above
is to deny connectivity by isolating a single vertex. Much less is
known, however, for the case $s<n$. It seems that even if \textsc{Breaker}
is able to isolate a vertex in a constant number of moves, it does
little to prevent \textsc{Maker} from winning the $s$-component game
for $s=\Omega\left(n\right)$.

\subsection{\label{sec:our-results}Our results}

Instead of considering the threshold bias $b^{*}$, we shift the focus
to the maximal component size $s$ achievable by \textsc{Maker} in
the $\left(1:b\right)$ game, for a given bias $b$. Let us denote
this quantity by $s_{b}^{*}\left(G\right)$, and let, for $d\ge3$,
\[
s_{b}^{*}\left(n,d\right)=\max\left\{ s_{b}^{*}\left(G\right):\mbox{\ensuremath{G} is a \ensuremath{d}-regular graph on \ensuremath{n} vertices}\right\} .
\]

For $b\ge2d-2$, \textsc{Breaker} can immediately isolate each edge
claimed by \textsc{Maker} in the $\left(1:b\right)$ game, so trivially
$s_{b}^{*}\left(G\right)=2$. Furthermore, \textsc{Breaker} can do
something similar while $b>d-2$, as the following proposition shows:
\begin{prop}
\label{prop:breaker-trivial-win}For any positive $k$, $s_{d-2+k}^{*}\left(n,d\right)\le2\left\lceil d/k\right\rceil $.
\end{prop}
In the $\left(1:d-2\right)$ game, \textsc{Breaker} can still restrict
the size of \textsc{Maker}'s connected components.
\begin{thm}
\label{thm:breaker-win}$s_{d-2}^{*}\left(n,d\right)\le\alpha_{d}+\beta_{d}\log n$,
where $\alpha_{d}$ and $\beta_{d}$ depend only on $d$.\end{thm}
\begin{rem*}
Our proof yields $\alpha_{d}=O\left(d^{2}\right)$ and $\beta_{d}=O\left(d/\log\log d\right)$.
\end{rem*}
The proof of Theorem~\ref{thm:breaker-win} relies on the following
combinatorial lemma, which may be of independent interest.
\begin{lem}
\label{lem:short-orientation}Let $G$ be a graph on $n$ vertices
with minimal degree $\delta\ge3$. Then, there exists an orientation
$D$ of $G$ such that every vertex has a positive out-degree and
all simple directed paths in $D$ are of length at most $\chi\left(G\right)+\kappa_{\delta}\log n$,
where $\kappa_{\delta}=O\left(1/\log\log\delta\right)$.\end{lem}
\begin{rem*}
Note that $s_{b}^{*}\left(T_{b+1}\left(k\right)\right)\ge k$, where
$T_{r}\left(k\right)$ is the complete $r$-ary tree%
\footnote{That is, every non-leaf vertex has $r$ children.%
} with $k$ levels, since \textsc{Maker} can easily build a path from
the root to some leaf. Completing $T_{d-1}\left(k\right)$ to a $d$-regular
graph on $n=\left(d-1\right)^{k}$ vertices thus shows that $s_{d-2}^{*}\left(n,d\right)\ge\log_{d-1}n$.
\end{rem*}
To complement Theorem~\ref{thm:breaker-win}, we prove that in the
$\left(1:d-3\right)$ game on almost every graph, \textsc{Maker} can
already build a very large connected component.
\begin{thm}
\label{thm:maker-win}Let $\mathcal{G}_{n,d}$ be the random $d$-regular
graph on $n$ vertices, where $d\ge3$. Then, $s_{d-3}^{*}\left(\mathcal{G}_{n,d}\right)\ge\epsilon_{d}n$
almost surely, where $\epsilon_{d}>0$ depends only on $d$. In particular,
\textup{$s_{d-3}^{*}\left(n,d\right)\ge\epsilon_{d}n$.}\end{thm}
\begin{rem*}
A quick calculation shows that $\epsilon_{d}\ge\mathrm{poly}\left(1/d\right)$.
\end{rem*}
\begin{figure}[h]
\noindent \begin{centering}
\subfloat[$\ensuremath{d=\Theta\left(1\right)}$]{$s_{b}^{*}\left(n,d\right)=\begin{cases}
\Theta\left(1\right), & b>d-2;\\
\Theta\left(\log n\right), & b=d-2;\\
\Theta\left(n\right), & b<d-2.
\end{cases}$

}\hspace{4em}\subfloat[$\ensuremath{d=O\left(\textrm{poly}\log n\right)}$]{$s_{b}^{*}\left(n,d\right)=\begin{cases}
\Theta\left(\textrm{poly}\log n\right), & b\ge d-2;\\
\Theta\left(n/\textrm{poly}\log n\right), & b<d-2.
\end{cases}$\vphantom{$s_{b}^{*}\left(n,d\right)=\begin{cases}
\Theta\left(1\right), & b>d-2;\\
\Theta\left(\log n\right), & b=d-2;\\
\Theta\left(n\right), & b<d-2.
\end{cases}$}

}
\par\end{centering}

\caption{\label{fig:phase-transition}Phase transition phenomenon at $b=d-2$}
\end{figure}

When $d$ is at most polylogarithmic in $n$, Theorems~\ref{thm:breaker-win}
and~\ref{thm:maker-win} show a phase transition phenomenon that
occurs at $b=d-2$; instead of tiny, polylogarithmic-sized components,
\textsc{Maker} is suddenly able to build a giant, almost linear-sized
component. When $d$ is constant, we even have a double-jump --- from
constant, through logarithmic, to linear-sized components. This is
summarized in Figure~\ref{fig:phase-transition}.

This behavior is somewhat consistent with the so-called random graph
intuition in positional games: oftentimes, the outcome of a game between
two intelligent players is the same as the outcome of that game between
two players acting randomly. Consider the bond percolation with parameter
$p$ (i.e., each edge is deleted independently with probability $1-p$).
It is known (see, e.g.,~\cite{ABS-04,NP-10}) that for well-expanding
$d$-regular graphs, where $d$ is constant, the size of the largest
connected component has a double-jump at $p=\frac{1}{d-1}$ --- it
is linear for $p\ge\frac{1+\epsilon}{d-1}$, logarithmic for $p\le\frac{1-\epsilon}{d-1}$,
and $\Theta\left(n^{2/3}\right)$ for $p=\frac{1}{d-1}$.

Although the sizes of the components are different, both the bond
percolation and $s_{b}^{*}\left(n,d\right)$ have a sharp threshold
at the same point, since in a random play of a $\left(1:b\right)$
game, \textsc{Maker} gets each edge with probability $\frac{1}{1+b}=\frac{1}{d-1}$.%
\footnote{This random graph intuition is not a formal argument, so we allow
ourselves to neglect the dependence of these random choices.%
}

\subsection{Notation}

We use standard Graph Theory terminology, and in particular use the
following:

For a given graph $G$ we denote by $V(G)$ and $E(G)$ the set of
its vertices and the set of its edges, respectively. We often just
use $V$ and $E$, when there is no chance of confusion. For two disjoint
sets of vertices $A,B\subseteq V$ we denote by $E(A,B)$ the set
of all edges $(a,b)\in E$ with $a\in A$ and $b\in B$. For a connected
component $S$ in \textsc{Maker}'s graph, and for an edge $e\in E$
we say that $e$ is \emph{incident} to $S$ if at least one of its
endpoints belongs to $S$; if both endpoints of $e$ belong to $S$,
we say that $e$ is \emph{inside }$S$.

When $G$ is a directed graph, we say that a vertex $v$ is \emph{reachable}
from a vertex $u$ if there is a directed path in $G$ from $u$ to
$v$.

An unclaimed edge is called \emph{free}. The act of claiming one free
edge by one of the players is called a \emph{step}. \textsc{Maker}'s
$m$ (\textsc{Breaker}'s $b$) successive steps are called a \emph{move}.
A \emph{round} in the game consists of one move of the first player,
followed by one move of the second player. Whenever \textsc{Maker}
claims a free edge, it becomes part of some connected component of
his; we then say he \emph{touched} that component. If a connected
component in \textsc{Maker}'s graph has at least one free edge adjacent
to it, we say it is a \emph{live} component.

\medskip{}

As mentioned before, if one of the players has a winning strategy
as a second player, he can use it to obtain a winning strategy as
a first player. Hence, when we describe \textsc{Maker}'s strategy
we assume that he is the second player, implying that under the described
conditions he can win as either a first or a second player. The same
goes for \textsc{Breaker}'s strategy.

\section{\label{sec:maker-strategy}\textsc{Maker}'s strategy}

Throughout this section we assume that the first player is \textsc{Breaker}.\medskip{}

In this section we describe and analyze a very basic strategy for
\textsc{Maker}, to which we refer throughout the paper as \emph{the
tree strategy}.\textsc{ Maker}'s goal is to build a component of size
$s$, and his strategy is to build a single connected component $T$.
He starts from a single arbitrary vertex $r$, and in every move he
adds a new vertex to $T$ by claiming a free edge $e\in E\left(T,V\setminus T\right)$.
If all edges in $E\left(T,V\setminus T\right)$ have already been
claimed by \textsc{Breaker}, and \textsc{Maker'}s component is of
size strictly less than $s$, he forfeits the game. Note that indeed
$T$ is a tree throughout the game.
\begin{defn*}
\label{def:isoperimetric-profile}Let $G=\left(V,E\right)$ be a graph
on $n$ vertices. For an integer $k=1,2,\ldots,\left\lfloor n/2\right\rfloor $,
we define 
\[
\Psi_{E}\left(G,k\right)=\min\left\{ \frac{\left|E\left(S,V\setminus S\right)\right|}{\left|S\right|}:S\subseteq V,1\le\left|S\right|\le k\right\} .
\]
Considered as a function of $k$, i.e., when $G$ is fixed, $\Psi_{E}$
is sometimes called the \emph{edge isoperimetric profile}.
\end{defn*}
The next proposition shows that if the graph has good expanding properties,
then \textsc{Breaker} cannot separate $T$ from $V\setminus T$ unless
$T$ is large enough.
\begin{prop}
\label{prop:tree-strategy-works}Assume $\Psi_{E}\left(G,k\right)>b$.
Then \textsc{Maker} is able to carry out the tree strategy for at
least $k$ rounds in the $\left(1:b\right)$ game on the graph $G$.\end{prop}
\begin{proof}
Consider the moment before \textsc{Maker}'s $j$th move for some $1\le j\le k$.
We have $\left|T\right|=j\le k$ and thus $\left|E\left(T,V\setminus T\right)\right|>\left|T\right|b=jb$.
During $j$ moves, \textsc{Breaker} could have claimed at most $jb$
edges, so some edge of $E\left(T,V\setminus T\right)$ is still available
for \textsc{Maker} to claim.
\end{proof}
Since we only need \textsc{Maker}'s connected component to span a
constant fraction of the graph, we can make use of the following result
on the edge expansion of small sets in the random $d$-regular graph:
\begin{lem}[\cite{HLW-06}, Theorem~4.16]
\label{lem:small-sets-expansion}Let $d\ge3$ be an integer and let
$\delta>0$. Then there exists $\epsilon=\epsilon\left(d,\delta\right)>0$
such that $\Psi_{E}\left(\mathcal{G}_{n,d},\epsilon n\right)>d-2-\delta$
almost surely.
\end{lem}
Taking $\delta=1$ in Lemma~\ref{lem:small-sets-expansion} and employing
the tree strategy yields Theorem~\ref{thm:maker-win}.
\begin{rem}
\label{rem:doubly-biased-maker}Lemma~\ref{lem:small-sets-expansion}
is strong enough to render the tree strategy effective also in the
doubly-biased $\left(m:b\right)$ game, as long as $b/m<d-2$. The
proof is completely analogous to the proof of Proposition~\ref{prop:tree-strategy-works}
when \textsc{Maker} is the first player, and very simple adjustments
are needed when \textsc{Breaker} starts.
\end{rem}

\section{\label{sec:breaker-strategy}\textsc{Breaker}'s strategy}

Throughout this section we assume that the first player is \textsc{Maker}.

\subsection{\label{sec:reactive}Reactive strategies}
\begin{defn*}
\label{def:reactive}A strategy of \textsc{Breaker} is called \emph{reactive}
if the following holds: in each of his steps, if the connected component
last touched by \textsc{Maker} is live, \textsc{Breaker} claims a
free edge incident to it.
\end{defn*}
Note that there can be many reactive strategies for \textsc{Breaker},
varying in the way that he chooses which free edge to claim among
those that are incident to \textsc{Maker}'s last touched component.
In this paper, we limit ourselves to reactive \textsc{Breaker} strategies;
this allows \textsc{Breaker} to control the number of free edges incident
to \textsc{Maker}'s connected components, as the following claim shows:
\begin{claim}
\label{clm:reactive}Let $b,d$ be positive integers and let $G$
be a $d$-regular graph. If \textsc{Breaker} uses a reactive strategy,
then throughout the $\left(1:b\right)$ Maker--Breaker game played
on the edge set of $G$, at the beginning of each round every connected
component $S$ in \textsc{Maker}'s graph is incident to at most $\left(d-2-b\right)\left|S\right|+b+2$
free edges.\end{claim}
\begin{proof}
The claim trivially holds at the beginning of the game, as every connected
component is a single vertex, and vertex degrees in $G$ are all equal
to $d$. In every move, \textsc{Maker} either:
\begin{enumerate}[label=(\emph{\alph*})]
\item claims an edge inside some component; or
\item merges two connected components, i.e., claims a free edge between
two connected components $S_{1}$ and $S_{2}$, creating a new connected
component $S$ of size $\left|S\right|=\left|S_{1}\right|+\left|S_{2}\right|$.
\end{enumerate}
In the first case, the claim trivially holds no matter how \textsc{Breaker}
plays. In the second case, $S_{i}$ (for $i=1,2$) was incident before
\textsc{Maker}'s move to at most $\left(d-2-b\right)\left|S_{i}\right|+b+2$
free edges. As \textsc{Maker} has just claimed an edge incident to
both $S_{1}$ and $S_{2}$, after \textsc{Maker}'s move at most $\left(d-2+b\right)\left|S\right|+2\left(b+2\right)-2$
free edges are incident to the merged component $S$. \textsc{Breaker}
in his next move claims $b$ of these edges (or simply all of them,
if there are less than that), leaving at most $\left(d-2-b\right)\left|S\right|+2\left(b+2\right)-2-b=\left(d-2-b\right)\left|S\right|+b+2$
free edges incident to $S$, so the claim still holds.
\end{proof}
We can now prove Proposition~\ref{prop:breaker-trivial-win}.
\begin{proof}[Proof of Proposition~\ref{prop:breaker-trivial-win}]
By Claim~\ref{clm:reactive} we get that by using any reactive strategy,
\textsc{Breaker} can make sure that every component $S$ in \textsc{Maker}'s
graph will have at most $-k\left|S\right|+k+d$ free edges incident
to it. In particular, $k+d>k\left|S\right|$ for any live component
$S$, or equivalently $|S|<\left(d/k\right)+1$. This last inequality
may be rewritten as $|S|\leq\left\lceil d/k\right\rceil $. Since
every component in \textsc{Maker}'s graph was created by merging two
live components, the result follows.\end{proof}
\begin{rem*}
For fixed $d$ and $k$, the bound of Proposition~\ref{prop:breaker-trivial-win}
is tight for large enough $n$, via the tree strategy on $\mathcal{G}_{n,d}$.
\end{rem*}

\begin{rem}
\label{rem:doubly-biased-breaker}Reactive strategies are effective
for \textsc{Breaker} also in the doubly-biased $\left(m:b\right)$
game, as long as $b/m>d-2$. A proof very similar to the one above
shows that \textsc{Breaker} can limit \textsc{Maker} in the $\left(m:m\left(d-2\right)+k\right)$
game to connected components of size at most $\left(m+1\right)\left\lceil md/k\right\rceil $.
\end{rem}

\subsection{\label{sec:breaker-vs-tree}Playing against the tree strategy}

Before presenting a full-fledged strategy for \textsc{Breaker} in
the $\left(1:d-2\right)$ game, let us first consider a simplified
version of it, which remains effective as long as \textsc{Maker} adheres
to the tree strategy of Section~\ref{sec:maker-strategy}. Taking
Claim~\ref{clm:reactive} one step further, \textsc{Breaker} needs
to make sure that, before \textsc{Maker}'s tree $T$ grows too much,
the only free edges incident to it will be edges inside $T$. This
gives rise to the following definition:
\begin{defn*}
\label{def:self-colliding-path}Let $G$ be a graph. A simple path
$p=v_{1}v_{2}\cdots v_{k}$ in $G$ is called \emph{self-colliding}
if $v_{k}$ is adjacent to some $v_{i}$ for $1\le i\le k-2$. We
could also view $p$ as a simple path $v_{1}v_{2}\cdots v_{i-1}$,
which we call the \emph{tail}, leading to a simple cycle $v_{i}v_{i+1}\cdots v_{k}v_{i}$,
which we call the \emph{body}.%
\footnote{It is possible that the tail is empty, that is, $p$ is a cycle of
length $k$.%
}
\end{defn*}
We use the following variation of the Moore bound on the girth of
graphs with minimum degree $k$.
\begin{lem}
\label{lem:Moore-variant}Let $G$ be a graph on $n$ vertices with
minimum degree $\delta\left(G\right)\ge k$. Then, for every edge
$\left(u,v\right)\in E$, there is a self-colliding path $p$ starting
with $\left(u,v\right)$ of length at most $2\left\lceil \log_{k-1}n\right\rceil $.
In particular, $g\left(G\right)\le2\left\lceil \log_{k-1}n\right\rceil $.
Moreover, the distance along $p$ from $u$ to every body vertex is
at most $\left\lceil \log_{k-1}n\right\rceil $. \end{lem}
\begin{proof}
The number of non-backtracking walks of length $j+1$ starting with
the edge $(u,v)$ is $\left(k-1\right)^{j}$. Since the graph only
has $n$ vertices, there exist two distinct non-backtracking walks
of lengths $i+1$ and $j+1$ ending at the same vertex, where $i\le j\le\left\lceil \log_{k-1}n\right\rceil $.
Together, these walks form a (not necessarily simple) cycle of length
at most $2\left\lceil \log_{k-1}n\right\rceil $ passing through $v$.
Now take any simple subcycle of it to be $p$'s body and connect it
back to $u$ via a simple path.
\end{proof}
We now describe \textsc{Breaker}'s strategy. After \textsc{Maker}'s
first move, \textsc{Breaker} chooses arbitrarily one of the two vertices
\textsc{Maker} has just touched and denotes it by $u$. \textsc{Breaker}
then uses Lemma~\ref{lem:Moore-variant} to pick, for each neighbor
$v$ of $u$, a self-colliding path $p_{v}$ of length at most $2\left\lceil \log_{d-1}n\right\rceil $
beginning with the edge $\left(u,v\right)$. Note that the paths chosen
for two neighbors $v,v'$ are not necessarily disjoint.

Now \textsc{Breaker}'s strategy is to allow \textsc{Maker} to claim
only edges from $P=\cup\left\{ p_{v}:\left(u,v\right)\in E\right\} $;
this would limit the size of \textsc{Maker}'s connected component
to be at most $\left|P\right|\le2d\left\lceil \log_{d-1}n\right\rceil $.
\begin{prop}
\label{prop:counter-tree-strategy-works}In the $\left(1:d-2\right)$
game on $G$, if \textsc{Maker} follows the tree strategy, \textsc{Breaker}
is able to carry out the counter-strategy.\end{prop}
\begin{proof}
We show that \textsc{Breaker} can ensure that before every move of
\textsc{Maker}, the only free edges in $E\left(T,V\setminus T\right)$
are in $P$; thus, \textsc{Maker} must claim an edge of $P$, advancing
along some $p_{v}$. It is true at the beginning of the game as $T=\left\{ u\right\} $.
After \textsc{Maker} claims the edge $\left(v_{i-1},v_{i}\right)\in p_{v}$,
there are at most $d-2$ free edges incident to $v_{i}$ in $E\left(T,V\setminus T\right)\setminus P$,
since $\left(v_{i-1},v_{i}\right)$ has just been claimed and $\left(v_{i},v_{i+1}\right)\in P$.
\textsc{Breaker} can claim all of them (and, if necessary, some arbitrary
extra edges outside $P$). Thus, after getting a spanning tree $T\subset P$,
\textsc{Maker} forfeits.
\end{proof}
The counter-strategy is still effective when \textsc{Maker} builds
a forest with many trees, as long as one of the connected components
merged is always a single vertex; nevertheless, it breaks down when
\textsc{Maker} builds up many small trees and connects them one to
the other, avoiding getting to the collision at the end of the self-colliding
paths. \textsc{Breaker} could possibly deny a merge of two trees $T$
and $T'$ by forgoing the counter-strategy and claiming the free edge
between $T$ and $T'$, but this might let \textsc{Maker} escape from
the respective $P$ or $P'$.

\subsection{\label{sec:breaker-vs-general}Playing against any strategy}

We now describe a global strategy for \textsc{Breaker}, which copes
well with \textsc{Maker} merging connected components of any size.
Before starting the $\left(1:d-2\right)$ game, \textsc{Breaker} uses
Lemma~\ref{lem:short-orientation} to pick an orientation $D$ of
the graph $G$ such that every vertex has a positive out-degree and
all simple directed paths in $D$ are of length at most $d+\kappa_{d}\log n$.
Note that \textsc{Breaker} may as well reveal $D$ to \textsc{Maker}.

\bigskip{}

The strategy of \textsc{Breaker} goes as follows. Without loss of
generality we may assume that \textsc{Maker}'s strategy is always
to build a forest, since claiming an edge within a connected component
does not help \textsc{Maker}.%
\footnote{Formally, \textsc{Maker} claims edges inside his connected components
only when all remaining free edges are such; by this time, the outcome
of the game has already been determined.%
} Thus, on each move \textsc{Maker} merges two trees $T_{1}$ and $T_{2}$
to a single tree $T$ by claiming a free edge from $T_{1}$ to $T_{2}$.
\textsc{Breaker} then claims $d-2$ free edges according to the following
priorities:
\begin{enumerate}
\item $E\left(V\setminus T,T_{2}\right)$;
\item $E\left(V\setminus T,T_{1}\right)$;
\item $E\left(T,V\setminus T\right)$.
\end{enumerate}
In each step, \textsc{Breaker} claims an arbitrary free edge from
the set with the smallest index. If there is no free edge among these
sets, he just claims an arbitrary free edge.
\begin{claim}
\label{clm:directed-trees}Each tree $T$ in \textsc{Maker}'s graph
is a directed tree in $D$; that is, there is some $r\in T$ --- which
we denote by the \textit{root} of $T$ --- such that every vertex
in $T$ is reachable from $r$. Moreover, at the beginning of each
round (i.e., after \textsc{Breaker}'s move), no free edges enter $T\setminus\left\{ r\right\} $.\end{claim}
\begin{proof}
The claim is trivially true at the beginning of the game, as the initial
connected components are single vertices, so every vertex is the root
and only member of its own directed tree. Suppose now that \textsc{Maker}
merged $T_{1}$ and $T_{2}$, two trees with roots $r_{1}$ and $r_{2}$,
respectively, by claiming an edge from $T_{1}$ to $T_{2}$. By our
assumption, before the merge the only free edges entering $T_{1}$
and $T_{2}$ were into $r_{1}$ and $r_{2}$, respectively. Hence,
\textsc{Maker} must have claimed an edge into $r_{2}$. Clearly, the
merged component is a directed tree, and all vertices in $T_{1}\cup T_{2}$
are now reachable from $r_{1}$, which becomes the root of the new
tree. Furthermore, the in-degree of every vertex in $D$, and in particular
of $r_{2}$, is at most $d-1$, so \textsc{Breaker}'s preference towards
$E\left(V\setminus T,T_{2}\right)$ ensures that all the edges entering
$r_{2}$ are claimed after \textsc{Breaker}'s move (one by the merge
and all the rest by \textsc{Breaker}), and so all the free edges entering
the new tree enter its root.
\end{proof}
It is beneficial to classify \textsc{Maker}'s trees by the number
of free in-edges.
\begin{defn*}
\label{def:type}The \emph{type }of a tree $T$ in \textsc{Maker}'s
graph is the number of free edges in $E\left(V\setminus T,T\right)$.
\end{defn*}
By Claim~\ref{clm:directed-trees}, the type of a tree is bounded
by the in-degree of its root, so the possible types are $0,1,\ldots,d-1$.
Claim~\ref{clm:directed-trees} also enables us to partially order
the vertices in each tree, giving rise to the following definition:
\begin{defn*}
\label{def:height}Let $T$ be a tree in \textsc{Maker}'s graph. The
\emph{height} of a vertex $v\in T$, denoted by $h\left(v\right)$,
is the length of the (unique) path $r\leadsto v$ in $T$, where $r$
is the root of $T$; the height of an edge $\left(u,v\right)\in E\left(T,V\setminus T\right)$
is $h\left(u,v\right)=h\left(u\right)$; the height of $T$, denoted
$h\left(T\right)$, is the maximum height over all $v\in T$.
\end{defn*}
We wish to bound the size of \textsc{Maker}'s trees. By the choice
of $D$, we know that the trees are not too {}``high'', but we also
need to ensure they do not become too {}``wide''.

For this, we refine \textsc{Breaker}'s strategy a bit. In the tree
$T$ just created by \textsc{Maker}, \textsc{Breaker} claims in-edges
from highest to lowest, and then out-edges from lowest to highest.
In more detail, in each step \textsc{Breaker} claims an incoming free
edge $\left(x,y\right)\in E\left(V\setminus T,T\right)$ such that
$h\left(y\right)$ is maximal, if possible; otherwise he claims an
outgoing free edge $\left(x,y\right)\in E\left(T,V\setminus T\right)$
such that $h\left(x,y\right)=h\left(x\right)$ is minimal. In both
cases, ties are broken arbitrarily.

\textsc{Breaker}'s preference of claiming low out-edges gives the
following:
\begin{claim}
\label{clm:out-arcs-top-to-bottom}Let $T$ be a tree in \textsc{Maker}'s
graph. If the edge $e\in E\left(T,V\setminus T\right)$ was claimed
by \textsc{Breaker}, then $h\left(e'\right)\ge h\left(e\right)$ for
every free edge $e'\in E\left(T,V\setminus T\right)$.\end{claim}
\begin{proof}
Note first that if \textsc{Breaker} has claimed an edge $(u,v)\in E(T,V\setminus T)$
for some tree $T$, then from that point until the end of the game
$u$ will only belong to trees of type zero. Indeed, according to
his strategy, \textsc{Breaker} has claimed $(u,v)$ only since there
were no free edges entering $T$, so at that point $T$ is of type
zero. Furthermore, by Claim~\ref{clm:directed-trees} we have that
at any point until the end of the game $u$ will only belong to trees
rooted at $T$'s root, implying that they will be of type zero as
well. Therefore, Claim~\ref{clm:out-arcs-top-to-bottom} trivially
holds when the type of $T$ is positive, since that implies that \textsc{Breaker}
has claimed only edges entering $T$. We thus assume $T$ is of type
zero.

At the moment \textsc{Breaker} claims $e$, there is no edge lower
than $e$ among all edges in $E\left(T,V\setminus T\right)$. In subsequent
rounds, the only changes to $E\left(T,V\setminus T\right)$ (and to
$T$) are when \textsc{Maker} claims some edge $e'$ from $T$ to
another tree $T'$. The height of all vertices of $T'$ in the merged
tree, and thus also of all new edges in $E\left(T,V\setminus T\right)$,
is at least $h\left(e'\right)+1>h\left(e\right)$.
\end{proof}
Recall that in the counter-strategy to the tree strategy, \textsc{Breaker}
only allowed \textsc{Maker} to pursue self-colliding paths, so \textsc{Maker}'s
final component consisted of $d$ paths $p_{v}$ sharing a root vertex.
Here, similarly, \textsc{Breaker}'s strategy allows \textsc{Maker}
to extend every free edge in $E\left(T,V\setminus T\right)$ to a
directed path. This motivates the following definition of width:
\begin{defn*}
\label{def:width}Let $T$ be a tree in \textsc{Maker}'s graph. The
\emph{$i$-width }of $T$, denoted $w_{i}\left(T\right)$, is the
number of vertices in $T$ of height $i$ plus the number of free
edges in $E\left(T,V\setminus T\right)$ of height strictly smaller
than $i$. The width of $T$, denoted $w\left(T\right)$, is the maximum
$i$-width in $T$, taken over $i=0,1,\ldots,h\left(T\right)$.
\end{defn*}
We are ready to prove the following proposition, which implies Theorem~\ref{thm:breaker-win}
since $\left|T\right|\le1+h\left(T\right)\cdot w\left(T\right)$.
\begin{prop}
\label{prop:width-bound}Let $T$ be a tree of type $t$ in \textsc{Maker}'s
graph. Then, 
\[
w\left(T\right)\le\begin{cases}
d-t, & 1\le t\le d-1;\\
2d-2, & t=0.
\end{cases}
\]
\end{prop}
\begin{proof}
We prove this by induction on the number of rounds in the game. The
proposition holds for trivial trees. Assume \textsc{Maker} merges
trees $T_{1}$ and $T_{2}$ of types $t_{1}$ and $t_{2}$, respectively,
by claiming the edge $\left(u,v\right)$, where $v$ is the root of
$T_{2}$. Then, the merged tree $T$ has type $t=\max\left(0,t_{1}+t_{2}-d+1\right)$
after \textsc{Breaker}'s move. Note that necessarily $t_{2}>0$.

The vertices of $T_{1}$ maintain their height in $T$; vertices that
had height $j$ in $T_{2}$, now have height $h\left(u\right)+1+j$
in $T$. For $i\le h\left(u\right)$, we have $w_{i}\left(T\right)=w_{i}\left(T_{1}\right)\le w\left(T_{1}\right)$;
for $i>h\left(v\right)$, the now-claimed edge $\left(u,v\right)$
no longer counts for the $i$-width of $T$, so
\begin{equation}
w_{i}\left(T\right)=w_{i}\left(T_{1}\right)-1+w_{i-h\left(u\right)-1}\left(T_{2}\right)\le w\left(T_{1}\right)+w\left(T_{2}\right)-1.\label{eq:level-width}
\end{equation}

If $t_{1}>0$ then, by the induction hypothesis, $w\left(T_{1}\right)\le d-t_{1}$
and $w\left(T_{2}\right)\le d-t_{2}$, so 
\[
w\left(T\right)\le w\left(T_{1}\right)+w\left(T_{2}\right)-1\le d-t_{1}+d-t_{2}-1=d-t.
\]

If $t_{1}=0$ then $t=0$ too; by the induction hypothesis, $w\left(T_{1}\right)\le2d-2$
and $w\left(T_{2}\right)\le d-t_{2}\le d-1$. For $i\le h\left(u\right)$,
as before, we have $w_{i}\left(T\right)\le w\left(T_{1}\right)\le2d-2$;
for $i>h\left(u\right)$, assuming we show that $w_{i}\left(T_{1}\right)\le d$,
the same calculation as in~(\ref{eq:level-width}) yields $w_{i}\left(T\right)\le w_{i}\left(T_{1}\right)+w\left(T_{2}\right)-1\le d+\left(d-1\right)-1=2d-2$.

By the definition of $w_{i}\left(T_{1}\right)$, there exist a set
$U\subseteq T_{1}$ of vertices of height $i$ and a set $A\subseteq E\left(T,V\setminus T\right)$
of free edges of height less than $i$ such that $w_{i}\left(T_{1}\right)=\left|U\right|+\left|A\right|$.
For every vertex $x\in U$, pick a leaf $x'\in T_{1}$ reachable (in
$T_{1}$) from $x$. The out-degree of $x'$ in $D$ is positive,
so pick some edge $e=\left(x',y\right)\in E\left(D\right)$. If $y\in T_{1}$,
no one will ever claim $e$; otherwise, $e\in E\left(T,V\setminus T\right)$
so \textsc{Maker} has not yet claimed it. By Claim~\ref{clm:out-arcs-top-to-bottom},
neither did \textsc{Breaker} since $h\left(e\right)=h\left(x'\right)\ge h\left(x\right)=i>h\left(u\right)$
and $\left(u,v\right)\in E\left(T_{1},V\setminus T_{1}\right)$ was
free before \textsc{Maker}'s move. Altogether, we have a set $A'$
of $\left|A'\right|=\left|U\right|$ free edges coming out of $T_{1}$,
disjoint from $A$ since edge heights in $A'$ are all at least $i$.
By Claim~\ref{clm:reactive}, $T_{1}$ is incident to at most $d$
free edges, so $w_{i}\left(T_{1}\right)=\left|A\right|+\left|A'\right|=\left|A\cup A'\right|\le d$,
establishing the proposition.\end{proof}
\begin{rem*}
In the previous subsection, using the counter-strategy to the tree
strategy, \textsc{Breaker} could bound $w\left(T\right)$ by ensuring
that, besides a single vertex of degree $d$, the degrees of all vertices
in $T$ were at most two. With the strategy presented in this subsection,
\textsc{Breaker} cannot limit $w\left(T\right)$ by bounding the number
of forks in $T$, i.e., the number of vertices of out-degree at least
2. Indeed, already for $d=3$, there exists a positive out-degree
orientation $D$ of a cubic graph $G$ and a strategy for \textsc{Maker}
to build a tree $T$ with $\Omega\left(h\left(T\right)\right)$ forks
in a $\left(1:1\right)$ game on $G$.
\end{rem*}

\section{\label{sec:orientation}Short graph orientations}

In this section we discuss and prove Lemma~\ref{lem:short-orientation}.
We begin by introducing the following notation:
\begin{defn*}
\label{def:max-length}For a directed graph $D$, we denote by $l\left(D\right)$
the maximal length of a simple directed path in $D$. For an undirected
graph $G$ and $j\in\left\{ 0,1\right\} $, we denote by $l_{j}\left(G\right)$
the minimum of $l\left(D\right)$ over all orientations $D$ of $G$
such that every vertex has out-degree at least $j$.
\end{defn*}
The case $j=0$, that is, when we drop the positive out-degrees requirement,
was considered by (at least) four independent works.
\begin{thm}[Gallai~\cite{Gallai-68}--Hasse~\cite{Hasse-65}--Roy~\cite{Roy-67}--Vitaver~\cite{Vitaver-62}]
\label{thm:gallai-roy}For every graph $G$, $l_{0}\left(G\right)=\chi\left(G\right)$.
\end{thm}
We mention here only the easy side of the proof, which will be used
shortly. To see that $l_{0}\left(G\right)\le\chi\left(G\right)$,
color $G$ properly by the colors $\left\{ 1,2,\ldots,\chi\left(G\right)\right\} $
and orient each edge $\left\{ u,v\right\} $ from $u$ to $v$ iff
$u$'s color is greater than $v$'s.

\medskip{}
Returning to the case $j=1$, we cannot expect an orientation $D$
with positive out-degrees for which $l\left(D\right)$ is independent
of $n$. Indeed, when every vertex has a positive out-degree, $D$
surely contains a directed cycle, so $l_{1}\left(G\right)\ge g\left(G\right)$.
Known constructions of $d$-regular graphs of high girth (see, e.g.,~\cite{Bollobas-78,ES-63,LPS-88})
yield families of graphs of order $n$, chromatic number $\Omega\left(d/\log d\right)$
and girth $\Omega\left(\log_{d-1}n\right)$. Thus, our best hope would
be to show that $l_{1}\left(G\right)=O\left(\log n\right)$.

\medskip{}
The main idea of the proof that follows is this: we find in $G$ a
set of disjoint short cycles, which we orient cyclically, and we orient
the rest of the edges {}``towards'' the cycles. Lemma~\ref{lem:Moore-variant}
will assist us in showing that simple directed paths outside the cycles
are necessarily short.
\begin{proof}[Proof of Lemma~\ref{lem:short-orientation}]
Fix $k=\max\left(3,\left\lceil \log\delta/\log\log\delta\right\rceil \right)$
and set $\gamma_{\delta}=\left\lceil \log_{\delta-1}n\right\rceil ,\gamma_{k}=\left\lceil \log_{k-1}n\right\rceil $.

Let $\mathcal{C}$ be a maximal collection of nonadjacent induced
cycles of length at most $2\gamma_{k}$. That is, we begin with an
empty collection $\mathcal{C}=\varnothing$ and, as long as there
exists an induced cycle $C$ in $G$ of length $\left|C\right|\le2\gamma_{k}$
whose vertices have no neighbors among $V_{\mathcal{C}}$, the vertices
of cycles in $\mathcal{C}$, we add $C$ to $\mathcal{C}$. Note that
$\mathcal{C}$ is nonempty since the girth of $G$ is at most $2\gamma_{\delta}\le2\gamma_{k}$,
by Lemma~\ref{lem:Moore-variant} (or the Moore bound).

Fix any cyclic orientation of the cycles in $\mathcal{C}$, and orient
the edges of $E\left(V_{\mathcal{C}},V\setminus V_{\mathcal{C}}\right)$
into $\mathcal{C}$. All edges incident to $\mathcal{C}$ are thus
oriented, as the cycles in $\mathcal{C}$ are induced and nonadjacent.
Since no edges are leaving any cycle in $\mathcal{C}$, once we orient
the rest of the graph, any simple directed path can contain at most
$2\gamma_{k}$ vertices of $V_{\mathcal{C}}$, which form its suffix.

\medskip{}

We now fuse all the vertices of cycles in $\mathcal{C}$ to a single
vertex $s$. Let $G'=\left(V',E'\right)$ be the resulting graph;
make it simple by discarding loops and parallel edges incident to
$s$. 

For every vertex $v\in V'$ we denote its distance from $s$ by $\rho\left(v\right)$.
We claim that $\rho\left(v\right)\le1+\gamma_{\delta}$; indeed, $v$
is within distance $\gamma_{\delta}$ of some short cycle $C$ by
Lemma~\ref{lem:Moore-variant} (specifically, $v$ is within distance
$\gamma_{\delta}$ of any vertex on $C$), and $C$ either intersects
some cycle in $\mathcal{C}$, is adjacent to some cycle in $\mathcal{C}$,
or simply $C\in\mathcal{C}$, by the maximality of $\mathcal{C}$.

For $i\in\left\{ 1,2,\ldots,1+\gamma_{\delta}\right\} $, consider
the level set $V_{i}'=\left\{ v\in V':\rho\left(v\right)=i\right\} $
and the subgraph $G_{i}\subset G'$ it induces. As in the proof of
Theorem~\ref{thm:gallai-roy}, we orient edges between $V_{i}'$
and $V'_{i+1}$ {}``downwards'' (i.e., from $V'_{i+1}$ to $V'_{i}$).
This ensures all vertices except $s$ have a positive out-degree,
via shortest paths to $s$.

By definition, every edge either lies inside a level set or connects
two successive level sets. Therefore, it only remains to orient edges
between same height vertices, which will be done using Theorem~\ref{thm:gallai-roy}.
For $G_{1}$ we have $l_{0}\left(G_{1}\right)=\chi\left(G_{1}\right)\le\chi\left(G\right)$
since $G_{1}\subset G$. By the maximality of $\mathcal{C}$, for
all $i>1$, $G_{i}$ has no cycle of length at most $2\gamma_{k}$.
Apply Lemma~\ref{lem:Moore-variant} to deduce that $G_{i}$ cannot
have a subgraph with minimum degree $k$; in other words, $G_{i}$
is $\left(k-1\right)$-degenerate, and, in particular, $k$-colorable.

Altogether, we have an orientation $D'$ of $G'$ satisfying 
\[
l\left(D'\right)\le\sum_{i=1}^{1+\gamma_{\delta}}l_{0}\left(G_{i}\right)=\sum_{i=1}^{1+\gamma_{\delta}}\chi\left(G_{i}\right)\le\chi\left(G\right)+k\gamma_{\delta};
\]
combined with the orientation of edges incident to $\mathcal{C}$
defined above, we get an orientation $D$ of $G$ with positive out-degrees
and $l\left(D\right)\le\chi\left(G\right)+k\gamma_{\delta}+2\gamma_{k}$.
Therefore, 
\[
l_{1}\left(G\right)\le l\left(D\right)=\chi\left(G\right)+O\left(\log n/\log\log\delta\right),
\]
as the lemma states.
\end{proof}

\section{\label{sec:conclusion}Concluding remarks and open problems}

\paragraph{Component games on other graphs.}

For the sake of simplicity, we presented our results in this paper
only for regular graphs, but these stay put under the alternative
definition 
\[
s_{b}^{*}\left(n,d\right)=\max\left\{ s_{b}^{*}\left(G\right):\mbox{\ensuremath{G} is a graph on \ensuremath{n} vertices and \ensuremath{\mbox{\ensuremath{\Delta\left(G\right)\le}}d}}\right\} .
\]
It would be interesting to consider the component game on families
of graphs of unbounded maximum degree. For instance, the Maker--Breaker
component game on $\mathcal{G}_{n,p}$ is considered by~\cite{KM-*}.

\paragraph{Doubly-biased games.}

As shown in Remark~\ref{rem:doubly-biased-maker}, Theorem~\ref{thm:maker-win}
can be easily extended to the doubly-biased game $\left(m:b\right)$
for $b/m<\left(d-2\right)$; similarly, Remark~\ref{rem:doubly-biased-breaker}
extends Proposition~\ref{prop:breaker-trivial-win} to the $\left(m:b\right)$
game for $b/m>\left(d-2\right)$. However, the strategy presented
in Subsection~\ref{sec:breaker-vs-general} is inadequate in the
$\left(m:\left(d-2\right)m\right)$ game. Indeed, already for $m=2$,
there exists a positive out-degree orientation $D$ of a $d$-regular
graph $G$ and a strategy for \textsc{Maker} to build a connected
component $S$ of width $\Omega\left(d^{h\left(S\right)}\right)$
in a $\left(2:2d-4\right)$ game on $G$. The key step in \textsc{Maker}'s
strategy is to merge connected components by claiming two out-edges
entering the same vertex, nullifying Claims~\ref{clm:directed-trees}
and~\ref{clm:out-arcs-top-to-bottom}, and thus Proposition~\ref{prop:width-bound}
no longer holds.

We believe that not all hope is lost for \textsc{Breaker}.
\begin{conjecture}
\label{conj:doubly-biased-breaker-win}Let $G$ be a $d$-regular
graph on $n$ vertices, where $d\ge3$, and let $m$ be a positive
integer. Then, in the $\left(m:\left(d-2\right)m\right)$ game on
$G$, \textsc{Breaker} can force \textsc{Maker} to build only connected
components of size $o\left(n\right)$, perhaps polylogarithmic (or
even logarithmic) in $n$.
\end{conjecture}

\paragraph{Very large degrees.}

Our results in Subsection~\ref{sec:our-results} hold for any value
of $d$, but yield little for $d=\Omega\left(n\right)$.

In particular, for $G=K_{n+1}$ and for every $\epsilon>0$, we get
$s_{\left(1+\epsilon\right)n}^{*}\left(K_{n+1}\right)\le2/\epsilon$
from Proposition~\ref{prop:breaker-trivial-win} and $s_{\left(1-\epsilon\right)n}^{*}\left(K_{n+1}\right)\ge\epsilon n$
from Proposition~\ref{prop:tree-strategy-works} (note that $\Psi_{E}\left(K_{n+1},k\right)=n+1-k$
for every $1\le k\le n/2$); however, for $b=d=n$ we get the meaningless
bounds $0\le s_{n}^{*}\left(K_{n}\right)\le n$. A slightly better
upper bound would be $s_{n}^{*}\left(K_{n+1}\right)\le1+\left\lceil n/2\right\rceil $,
just because a $\left(1:b\right)$ game on any graph $G$ lasts $\left\lceil \left|E\left(G\right)\right|/\left(b+1\right)\right\rceil $
rounds.

It would be interesting to get a nontrivial bound on $s_{n}^{*}\left(K_{n+1}\right)$.

\paragraph{Very large components.}

Recall the proof of Theorem~\ref{thm:maker-win} in Section~\ref{sec:maker-strategy},
which combined the tree strategy with edge expansion via Proposition~\ref{prop:tree-strategy-works}.
How far can Proposition~\ref{prop:tree-strategy-works} push \textsc{Maker}?
Can \textsc{Maker} use it to build a connected component of size $\left\lfloor n/2\right\rfloor $?
The following upper bound on the Cheeger constant of regular graphs,
due to Alon~\cite{Alon-97}, says that this is only possible when
the bias is well below $d/2$.
\begin{thm}[\cite{Alon-97}]
For every $d$-regular graph $G$, $\Psi_{E}\left(G,\left\lfloor n/2\right\rfloor \right)\le d/2-\Omega(\sqrt{d})$.
\end{thm}
Proposition~\ref{prop:tree-strategy-works} poses a sufficient, but
obviously not a necessary, condition for the tree strategy to succeed.
It may be possible for \textsc{Maker} to build a connected component
of size $\left\lfloor n/2\right\rfloor $ via the tree strategy or
some other strategy, without relying on expansion.

\paragraph{Short orientations.}

The proof of Lemma~\ref{lem:short-orientation} shows that $l_{1}\left(G\right)\le\chi\left(G\right)+O\left(\log n/\log\log d\right)$.
On the other hand, $l_{1}\left(G\right)\ge g\left(G\right)$ and $l_{1}\left(G\right)\ge l_{0}\left(G\right)=\chi\left(G\right)$
and thus constructions of $d$-regular graphs of girth $\Omega\left(\log_{d-1}n\right)$
and chromatic number $\Omega\left(d/\log d\right)$ (see, e.g.,~\cite{Bollobas-78,ES-63,LPS-88})
demonstrate that sometimes $l_{1}\left(G\right)\ge\chi\left(G\right)+\Omega\left(\log n/\log d\right)$. 

We suspect the correct behavior of $l_{1}\left(G\right)$ is actually
the lower bound, as the following conjecture states.
\begin{conjecture}
\label{conj:short-orientation}Let $G$ be a $d$-regular graph on
$n$ vertices, where $d\ge3$. Then, $l_{1}\left(G\right)=\chi\left(G\right)+O\left(\log n/\log d\right)$.
\end{conjecture}
One can also ask about the value of $l_{j}\left(G\right)$ for $j>1$.

\subsection*{Acknowledgements}

The authors wish to thank Noga Alon, Asaf Ferber, Danny Hefetz, and
Michael Krivelevich for useful discussions and comments.

\end{document}